\documentclass[Alon1,singlecolor,11pt]{my-Alon}

\usepackage{amsthm,amssymb,amsmath}

\usepackage[english]{babel}
\usepackage{epigrafica}
\usepackage[LGR,OT1,T1]{fontenc} 
\usepackage{lmodern}
\usepackage{graphicx}
\usepackage{subfigure}
\usepackage{makeidx}
\usepackage{multicol}

\usepackage{changepage}
\usepackage{version}
\usepackage{url}
\usepackage{fancyhdr}

\newtheorem{theorem}{Theorem}

\newtheorem{proposition}{Proposition}

\renewcommand{\thesection}{\arabic{section}} 

\renewcommand{\thefigure}{\arabic{figure}} 

\begin{document}

\pagestyle{fancy}
\fancyhead{} 
\fancyhead[RO]{\textbf{The Lotka--Volterra Dynamical System and its Discretization}}
\fancyhead[LE]{\textbf{M\'{a}rcia Lemos-Silva and Delfim F. M. Torres}}
\fancyfoot{} 
\fancyfoot[LE,RO]{}
\fancyfoot[LO,RE]{}
\fancyfoot[CO,CE]{\thepage}


\thispagestyle{empty}

\begin{center}
\huge{\bf The Lotka-Volterra Dynamical System and its Discretization}
\end{center}

\let\thefootnote\relax\footnotetext{This is a preprint of a paper
whose final form is published as Chapter~19 in the book 
``\emph{Advanced Mathematical Analysis and Its Applications}'', 
Chapman \& Hall, 2023 (ISBN:~9781032481517).}

\begin{center}	
\textbf{M\'{a}rcia Lemos-Silva}\\
\texttt{marcialemos@ua.pt}\\
Center for Research and Development in Mathematics\\ 
and Applications (CIDMA), Department of Mathematics,\\ 
University of Aveiro, 3810-193 Aveiro, Portugal\\ 
\url{https://orcid.org/0000-0001-5466-0504}
\end{center}

\begin{center}	
\textbf{Delfim F. M. Torres}\\
\texttt{delfim@ua.pt}\\
Center for Research and Development in Mathematics\\ 
and Applications (CIDMA), Department of Mathematics,\\ 
University of Aveiro, 3810-193 Aveiro, Portugal\\ 
\url{https://orcid.org/0000-0001-8641-2505}
\end{center}


\noindent \textbf{Abstract}.
Dynamical systems are a valuable asset for the study
of population dynamics. On this topic, much has been done 
since Lotka and Volterra presented the very first continuous 
system to understand how the interaction between two species 
-- the prey and the predator -- influences the growth of both populations.
The definition of time is crucial and, among options, 
one can have continuous time and discrete time. The choice of a
method to proceed with the discretization of a continuous dynamical system
is, however, essential, because the qualitative behavior of the system is expected
to be identical in both cases, despite being two different temporal spaces.
In this work, our main goal is to apply two different discretization methods
to the classical Lotka--Volterra dynamical system: the standard progressive Euler's
method and the nonstandard Mickens' method. Fixed points and their stability are
analyzed in both cases, proving that the first method leads to dynamic
inconsistency and numerical instability, while the second is capable of keeping
all the properties of the original continuous model.\\

Mickens' discretization.\index{Mickens' discretization}


\section{Introduction to the Lotka--Volterra model}
\label{LST:section01}

Prey-predator equations\index{prey-predator model} intend to describe the dynamics 
of an ecological system where two species interact with each other. Alfred J. Lotka (1880--1949) 
introduced such equations in 1925 \cite{LST:lotka}; and Vito Volterra (1860-1940) studied them, 
independently \cite{LST:volterra}. For this reason, these equations are known as the Lotka--Volterra 
equations.\index{Lotka--Volterra model} On this topic, much has been
done since Lotka and Volterra presented the very first
continuous system to understand how the interaction
between two species -- the prey and the predator --
influences the growth of both populations 
\cite{LST:MR4509368,LST:MR4512320,LST:MR4509395}. 
For a review of some recent advances, we refer 
the reader to \cite{LST:review}. 

Here we consider the classical model of Lotka--Volterra, which is composed by two autonomous 
and nonlinear differential equations given by\index{Lotka--Volterra model}
\begin{equation}
\label{LST:lotka_volterra}
\begin{cases}
\dot{x} = \alpha x - \beta xy,\\
\dot{y} = -\delta y + \gamma xy,
\end{cases}
\end{equation}
where $x(t)$ and $y(t)$ represent the size at time $t$ of prey and predator populations, respectively. 
Moreover, all the parameters $\alpha$, $\beta$, $\gamma$, and $\delta$ are assumed to be positive. 

If the density of both species reach the zero value at any moment $t$, 
then they will remain there indefinitely, which represents the natural extinction of both species. 
The absence of prey leads to the extinction of predators since in that case $y(t)$ converges to 0 
when $t \rightarrow +\infty$. On the other hand, the absence of predators leads to exponential 
growth of prey, since $x(t) \rightarrow +\infty$ when $t \rightarrow +\infty$.

From an ecological point of view, population densities must always be nonnegative, 
restricting the system trajectories to $\mathbb{R}^2_+$. In fact, from the equations 
of system (\ref{LST:lotka_volterra}), we have 
\begin{equation*}
\begin{cases}
\left.\dot{x}\right|_{x = 0} = 0,\\
\left.\dot{y}\right|_{y = 0} = 0, 
\end{cases}
\end{equation*}
from which,	 according to Lemma~2 of \cite{LST:invariant}, we can conclude that the 
solution of the system is nonnegative, meaning that $\mathbb{R}^2_+$ 
is the invariant domain of the system. 

This system has equilibria at two different points: $p_1 = (0,0)$ and 
$p_2 = \left(\frac{\delta}{\gamma}, \frac{\alpha}{\beta}\right)$. To observe 
the approximate behavior of the solutions over time near these equilibrium points, 
we start by computing the Jacobian\index{Jacobian} 
matrix of the system, which is given by
\begin{equation}
\label{LST:jacobian_cont}
J(x,y) = 
\begin{pmatrix}
\alpha - \beta y & -\beta x \\ 
\gamma y & \gamma x - \delta
\end{pmatrix}.
\end{equation}

The Jacobian\index{Jacobian} matrix (\ref{LST:jacobian_cont}), 
evaluated at the equilibrium $(0,0)$, 
is given by
\begin{equation*}
J(0,0) = 
\begin{pmatrix}
\alpha & 0 \\ 
0 & - \delta
\end{pmatrix}.
\end{equation*}
The corresponding eigenvalues are $\lambda_1 = \alpha$ and $\lambda_2 = -\delta$ 
and, as $\alpha, \delta > 0$, it turns out that $(0,0)$ 
is a saddle point.\index{saddle point} In contrast, 
the Jacobian\index{Jacobian} matrix (\ref{LST:jacobian_cont}) 
evaluated at the coexistence equilibrium point $p_2$ is 
\begin{equation*}
J \left(\frac{\delta}{\gamma}, \frac{\alpha}{\beta}\right) 
= 
\begin{pmatrix}
0 & -\frac{\beta\delta}{\gamma} \\ 
\frac{\alpha\gamma}{\beta} & 0
\end{pmatrix},
\end{equation*}
for which eigenvalues are pure imaginary: $\lambda = \pm i\sqrt{\alpha\delta}$. 
This means that $p_2$ is a stable center in the linearized system. 
However, with this analysis, nothing can be concluded regarding the 
stability for the nonlinear system at this equilibrium.

To investigate the phase portrait of system (\ref{LST:lotka_volterra}), 
we start by drawing the two lines
\begin{gather*}
x = \frac{\delta}{\gamma} \quad,\quad y = \frac{\alpha}{\beta}.
\end{gather*}
By doing so, the first quadrant of the $xy$-plane is divided 
into four different regions, as shown in Figure~\ref{LST:regions}. 
\begin{figure}[ht!]
\centering
\includegraphics[scale = 0.6]{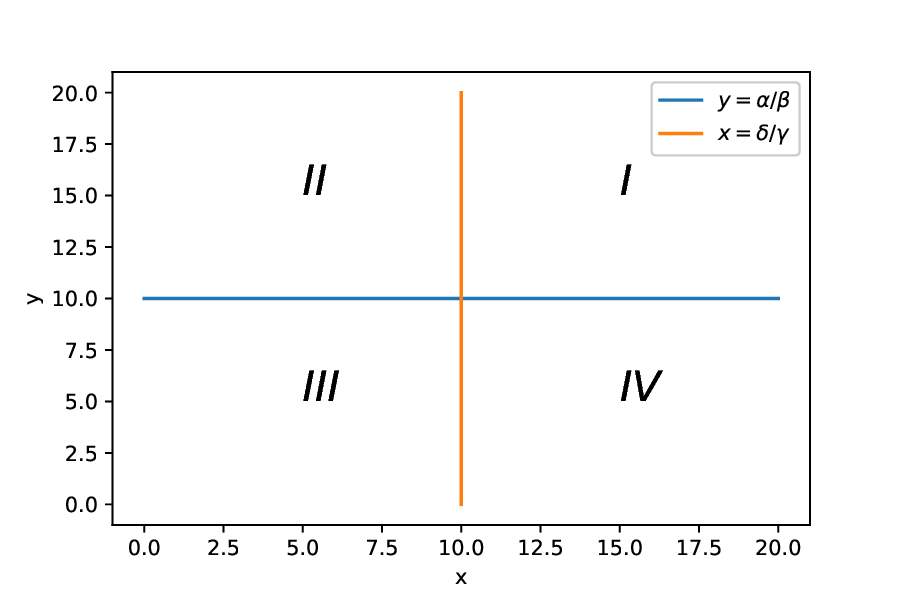} 
\caption{Regions defined by the lines $x = \frac{\delta}{\gamma}$ 
and $y = \frac{\alpha}{\beta}$, with $\alpha = 1$, 
$\beta = 0.1$, $\gamma = 0.075$, $\delta = 0.75$.}
\label{LST:regions}
\end{figure}

In each region, the signs of $\dot{x}$ and $\dot{y}$ determine 
the behavior of the solution of the system. By analyzing the 
equations of system (\ref{LST:lotka_volterra}), 
the following result holds. 

\begin{proposition}
\label{LST:trajectory_cont}
The trajectory $x$ of the system will
\begin{itemize}
\item decrease in region I and II, since $\dot{x} < 0$;
\item increase in region III and IV, since $\dot{x} > 0$.
\end{itemize}
Regarding the trajectory $y$ of the system, it will
\begin{itemize}
\item decrease in region II and III, since $\dot{y} < 0$;
\item increase in region I and IV, since $\dot{y} > 0$. 
\end{itemize}
\end{proposition}

Proposition~\ref{LST:trajectory_cont} suggests that the curve of the system 
in the phase plane will be counterclockwise around the equilibrium point $p_2$, 
but that it is not enough to conclude whether the trajectory spiral\index{spiral} 
towards $p_2$; spiral out, towards infinity; or it is a closed curve. 
Despite this, it has already been explained that this equilibrium is, in fact, 
a center in the nonlinear system, meaning the trajectories 
will be closed curves. This allows to write the following result. 

\begin{proposition}
\label{LST:closed}	
Except for those beginning at the equilibrium $p_2$ or at a coordinate axes,
every trajectory of the system is a closed orbit\index{orbit} 
that turns counterclockwise around the equilibrium point $p_2$. 
\end{proposition}

The behavior described by Proposition~\ref{LST:closed}	
can be seen in Figure~\ref{LST:phase}, 
for several different initial conditions.

\begin{figure}[!ht]
\centering
\includegraphics[scale = 0.4]{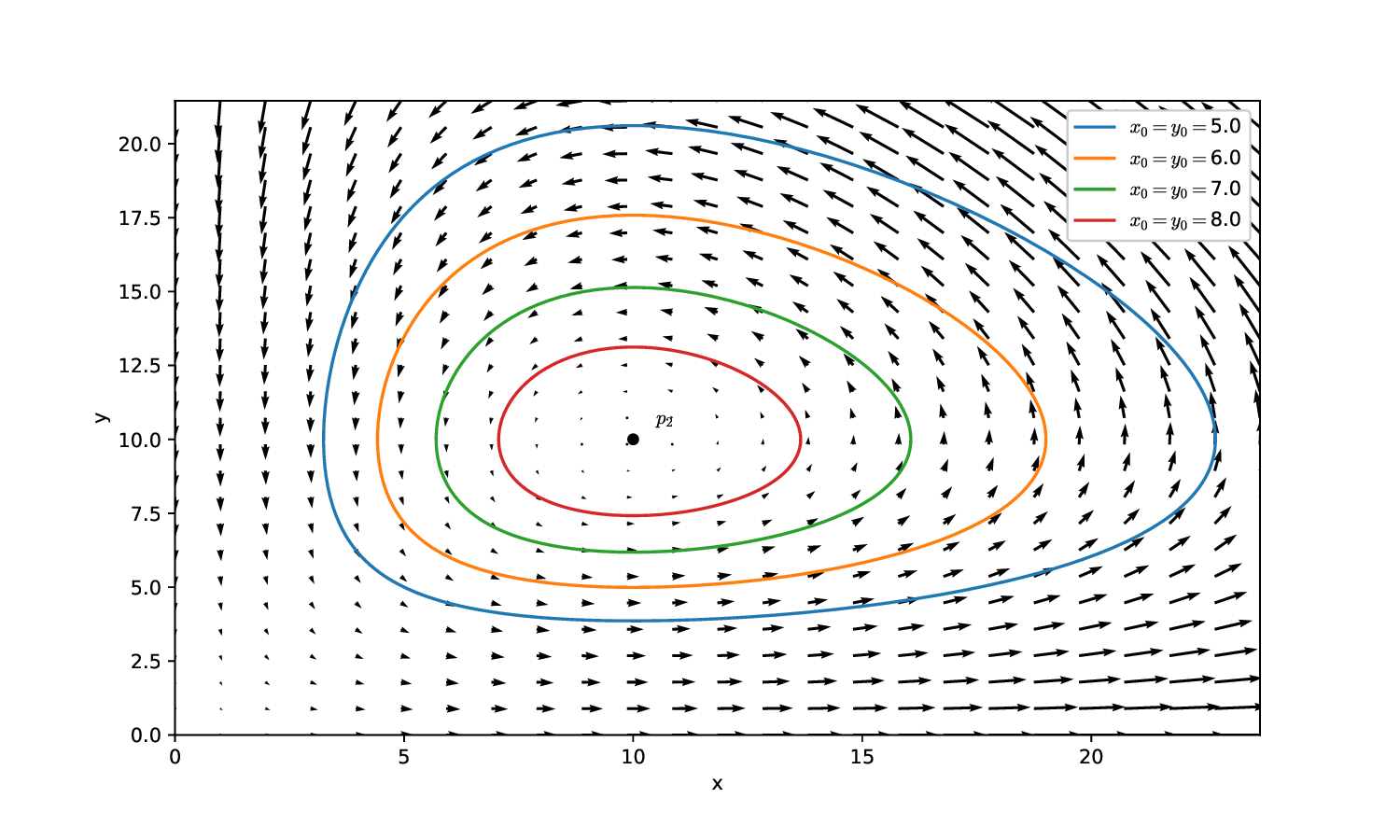} 
\caption{Phase portrait of system (\ref{LST:lotka_volterra}) with $\alpha = 1$, 
$\beta = 0.1$, $\gamma = 0.075$, and $\delta = 0.75$.}
\label{LST:phase}
\end{figure}

From Proposition~\ref{LST:closed}, it comes directly that the densities 
of predators and prey will oscillate periodically,\index{periodic} 
as can be seen in Figure~\ref{LST:oscillations}, with the amplitude 
and frequency of oscillations\index{oscillations} depending only 
on the considered initial conditions.

\begin{figure}[!ht]
\centering
\includegraphics[scale=0.45]{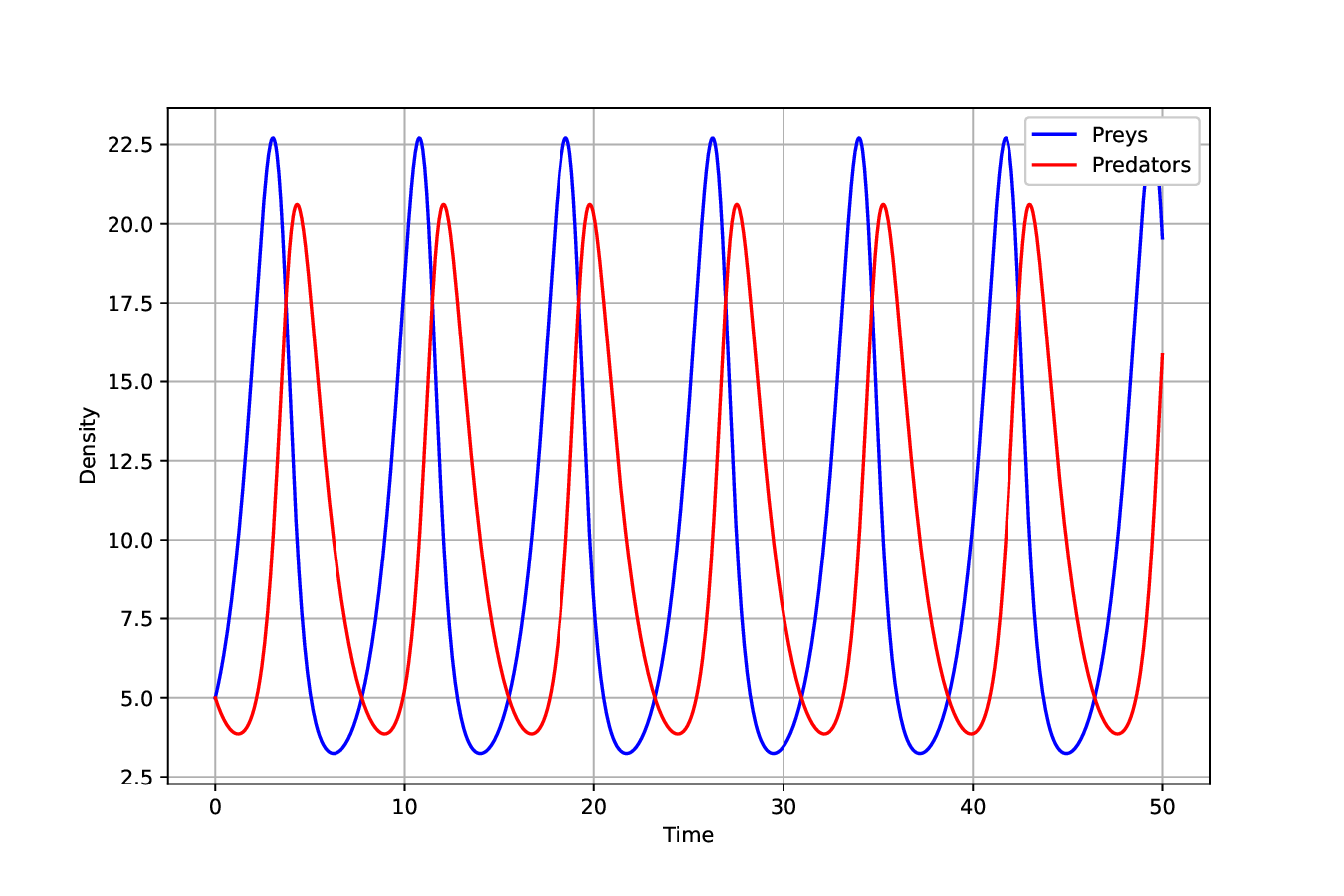} 
\caption{Oscillations of prey and predator densities for system (\ref{LST:lotka_volterra}) 
with $\alpha = 1$, $\beta = 0.1$, $\gamma = 0.075$, and $\delta = 0.75$.}
\label{LST:oscillations}
\end{figure}

All the results described so far are well-known.
In particular, both Propositions~\ref{LST:trajectory_cont} 
and \ref{LST:closed} can be found, e.g., in \cite{LST:lotka:theorems}. 
In Sections~\ref{LST:section02} and \ref{LST:section03} 
we provide new insights.
 

\section{Discretization by Euler's Method}
\label{LST:section02}

There are several methods for converting continuous systems 
into discrete counterparts. The most conventional way 
to do so is to implement a standard difference scheme,
the most classical one being the progressive 
Euler's method.\index{Euler's discretization}  
However, it is known that this method can raise several problems 
such as lack of dynamical consistency, even when applied 
to the simplest systems \cite{LST:consistent}. 
A discrete-time model is said to be dynamically consistent\index{dynamically consistent} 
with its continuous analog if they both exhibit the same dynamical behavior, 
namely the stability behavior of fixed points, bifurcation, and chaos. 
In \cite{LST:mickens:book}, Mickens points out that the fundamental reason 
for the existence of numerical instabilities is that discrete models 
have a larger parameter space than the corresponding differential equations: 
one has the step size $h$ as an additional parameter. Nevertheless, this step $h$
is, obviously, inherent to any discretization and any discrete dynamical system. 
Therefore, it is crucial to consider a numerical method that is able to overcome 
this setback during discretization.

Here we prove that Euler's method applied to the Lotka--Volterra 
model\index{Lotka--Volterra model} brings a discrete system 
that is not dynamically consistent with its continuous counterpart. 

Applying the progressive Euler's method\index{Euler's discretization} 
to both equations of system (\ref{LST:lotka_volterra}), we obtain that
\begin{equation}
\label{LST:euler}
\begin{cases}
x_{i+1} = x_i + h(\alpha x_i - \beta x_iy_i),\\
y_{i+1} = y_i + h(\gamma x_i y_i - \delta y_i),
\end{cases}
\end{equation}
where $h$ denotes the step size and $x_i$ and $y_i$ define the density 
of the prey and predators' populations, respectively, at time $i$. 

The fixed points\index{fixed point} of system (\ref{LST:euler}) are $p_1 = (0,0)$ 
and $p_2 = \left(\frac{\delta}{\gamma}, \frac{\alpha}{\beta}\right)$. 
To determine the nature of the fixed points,\index{fixed point} one must compute the 
Jacobian\index{Jacobian} matrix $J(x,y)$ of system (\ref{LST:euler}). 
This matrix is given by 
\begin{equation}
\label{LST:jacobian_euler}
J (x,y) = 
\begin{pmatrix}
-\beta hy + \alpha h + 1 & -\beta hx \\
\gamma hy & \gamma hx - \delta h + 1
\end{pmatrix}.
\end{equation}

Follows our first result.

\begin{theorem}
\label{LST:thm:01}
The fixed point\index{fixed point} $(0,0)$ of system \eqref{LST:euler} is 
\begin{itemize}
\item a saddle point\index{saddle point} if $h \in \left]0, \frac{2}{\delta}\right[$;
\item a source if $h \in \left]\frac{2}{\delta}, +\infty\right[$.
\end{itemize}
\end{theorem}

\begin{proof}
The Jacobian\index{Jacobian} matrix (\ref{LST:jacobian_euler}) 
evaluated at the fixed point\index{fixed point} $(0,0)$ is
\begin{equation*}
J(0,0) = 
\begin{pmatrix}
\alpha h + 1 & 0 \\ 
0 & -\delta h + 1
\end{pmatrix},
\end{equation*}
whose eigenvalues are $\lambda_1 = -\delta h + 1$ and $\lambda_2 = \alpha h + 1$. 
As all the parameters are positive, one can easily conclude that 
$\lvert \lambda_2 \rvert > 1$. On the other hand, $\lvert\lambda_1\rvert$ 
can either be greater or less that one. In particular, 
\begin{equation*}
\begin{aligned}
\lvert \lambda_1 \rvert < 1 
&\Leftrightarrow -\delta h + 1 < 1 \wedge -\delta h + 1 > -1\\ 
&\Leftrightarrow h > 0 \wedge h < \frac{2}{\delta},
\end{aligned}
\end{equation*}
and 
\begin{equation*}
\begin{aligned}
\lvert \lambda_1 \rvert > 1 
&\Leftrightarrow -\delta h + 1 > 1 \vee -\delta h + 1 < -1\\ 
&\Leftrightarrow h < 0 \vee h > \frac{2}{\delta}.
\end{aligned}
\end{equation*}
Therefore, the fixed point\index{fixed point} $(0,0)$ 
is a saddle point\index{saddle point} if 
$h \in \left]0, \frac{2}{\delta}\right[$ or a source if 
$h \in \left]\frac{2}{\delta}, +\infty\right[$. As $h$ 
is strictly positive, the condition $h < 0$ is not considered. 
\end{proof}

For both possibilities of Theorem~\ref{LST:thm:01}, 
the point $p_1 = (0,0)$ is unstable, which brings no major 
changes to what is obtained in the continuous case,
as described in Section~\ref{LST:section01}. We now
study what happens with the second fixed point $p_2$.

\begin{theorem}
\label{LST:thm:uf}
The fixed point\index{fixed point} $\left(\frac{\delta}{\gamma}, 
\frac{\alpha}{\beta}\right)$ is an unstable focus. 
\end{theorem}

\begin{proof}
The Jacobian\index{Jacobian} matrix (\ref{LST:jacobian_euler}) 
evaluated at the fixed point\index{fixed point} $\left(\frac{\delta}{\gamma}, 
\frac{\alpha}{\beta}\right)$ is given by 
\begin{equation*}
J\left(\frac{\delta}{\gamma}, \frac{\alpha}{\beta}\right) 
= \begin{pmatrix}
1 & -\frac{\beta\delta h}{\gamma} \\ \frac{\alpha\gamma h}{\beta} & 1
\end{pmatrix},
\end{equation*}
whose eigenvalues are the complex conjugates $\lambda = 1 \pm \sqrt{\alpha\delta}h$. 
As $\alpha$, $\delta$, and $h$ are strictly positive, 
it is clear that $\lvert \lambda \rvert > 1$, 
meaning that the fixed point\index{fixed point} $\left(\frac{\delta}{\gamma}, 
\frac{\alpha}{\beta}\right)$ is an unstable focus. 
\end{proof}

Theorem~\ref{LST:thm:uf} asserts that the orbits\index{orbit} of system 
(\ref{LST:euler}) near the fixed point\index{fixed point} $p_2$ will not be closed, 
but spirals\index{spiral} that spiral out toward infinity. 

Through a simple analysis of the system equations, taking into consideration 
the four regions defined in Figure~\ref{LST:regions}, it is possible to 
understand the direction of the solution in those regions.

\begin{theorem}
\label{LST:trajectory_euler}	
The trajectory $x$ of system (\ref{LST:euler}) will
\begin{itemize}
\item decrease in region I and II, i.e., $x_{i+1} < x_i$;
\item increase in region III and IV, i.e., $x_{i+1} > x_i$. 
\end{itemize}
Regarding the trajectory $y$ of system (\ref{LST:euler}), it will
\begin{itemize}
\item decrease in region II and III, i.e., $y_{i+1} < y_i$;
\item increase in region I and IV, i.e., $y_{i+1} > y_i$. 
\end{itemize}
\end{theorem}

\begin{proof}
We start by analyzing the trajectory of $x$ by looking 
to the first equation of system (\ref{LST:euler}). In 
regions I and II, $y_i > \frac{\alpha}{\beta}$. 
This means that 
\begin{equation*}
h(\alpha x_i - \beta x_iy_i) < 0,
\end{equation*}
and
\begin{equation*}
x_i + h(\alpha x_i - \beta x_iy_i) < x_i 
\Rightarrow x_{i+1} < x_i.
\end{equation*}	
On the other hand, in regions III and IV, 
we have $y_i < \frac{\alpha}{\beta}$. In this case,
\begin{equation*}
h(\alpha x_i - \beta x_iy_i) > 0,
\end{equation*}
from which we can conclude that
\begin{equation*}
x_i + h(\alpha x_i - \beta x_iy_i) > x_i 
\Rightarrow x_{i+1} > x_i.
\end{equation*}
Through the second equation of the system, by an analogous reasoning, 
the intended conclusions are obtained for $y$.
\end{proof}

Theorem~\ref{LST:trajectory_euler} implies
a counterclockwise displacement of the system. 
Precisely, it follows directly from 
Theorem~\ref{LST:trajectory_euler} that a trajectory 
of (\ref{LST:euler}) near the fixed point\index{fixed point} $p_2$ 
will spiral\index{spiral} in a counterclockwise direction, 
as can be seen in Figure~\ref{LST:diagrama_euler}.

\begin{figure}[ht!]
\centering
\includegraphics[scale = 0.55]{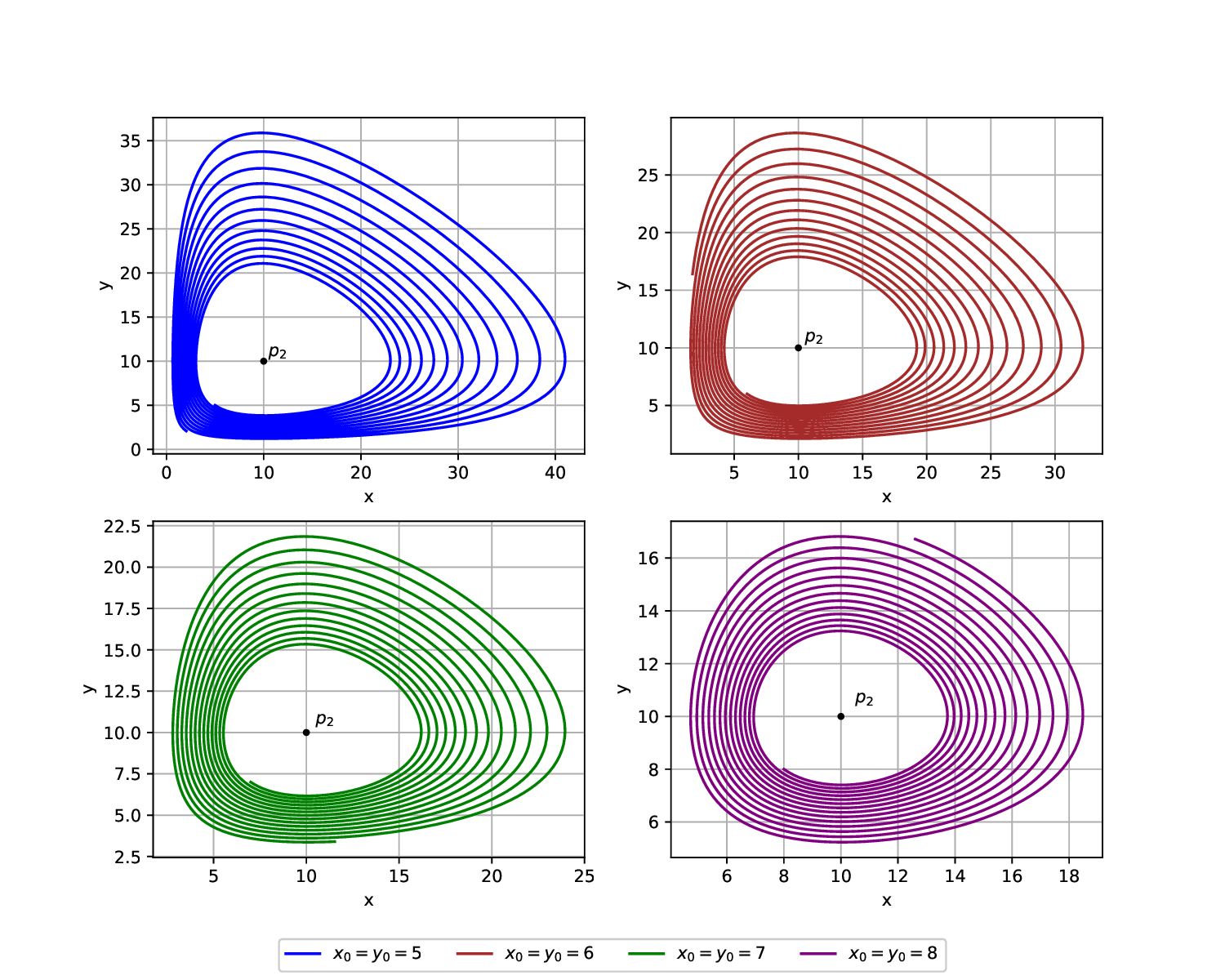} 
\caption{Trajectory of the system near the fixed point\index{fixed point} $p_2$ with 
$\alpha = 1$, $\beta = 0.1$, $\gamma = 0.075$, $\delta = 0.75$, and $h = 0.02$.}
\label{LST:diagrama_euler}
\end{figure}

Since the trajectories of system (\ref{LST:euler}) are not closed curves, 
they cease to be periodic\index{periodic} orbits.\index{orbit} 
Here, the trajectories are expansive, not converging to a particular 
fixed point, and the amplitude of the curves does not remain constant. 
In this case, the amplitude increases over time: see
Figure~\ref{LST:oscilacoes_euler}.

\begin{figure}[ht!]
\centering
\includegraphics[scale = 0.45]{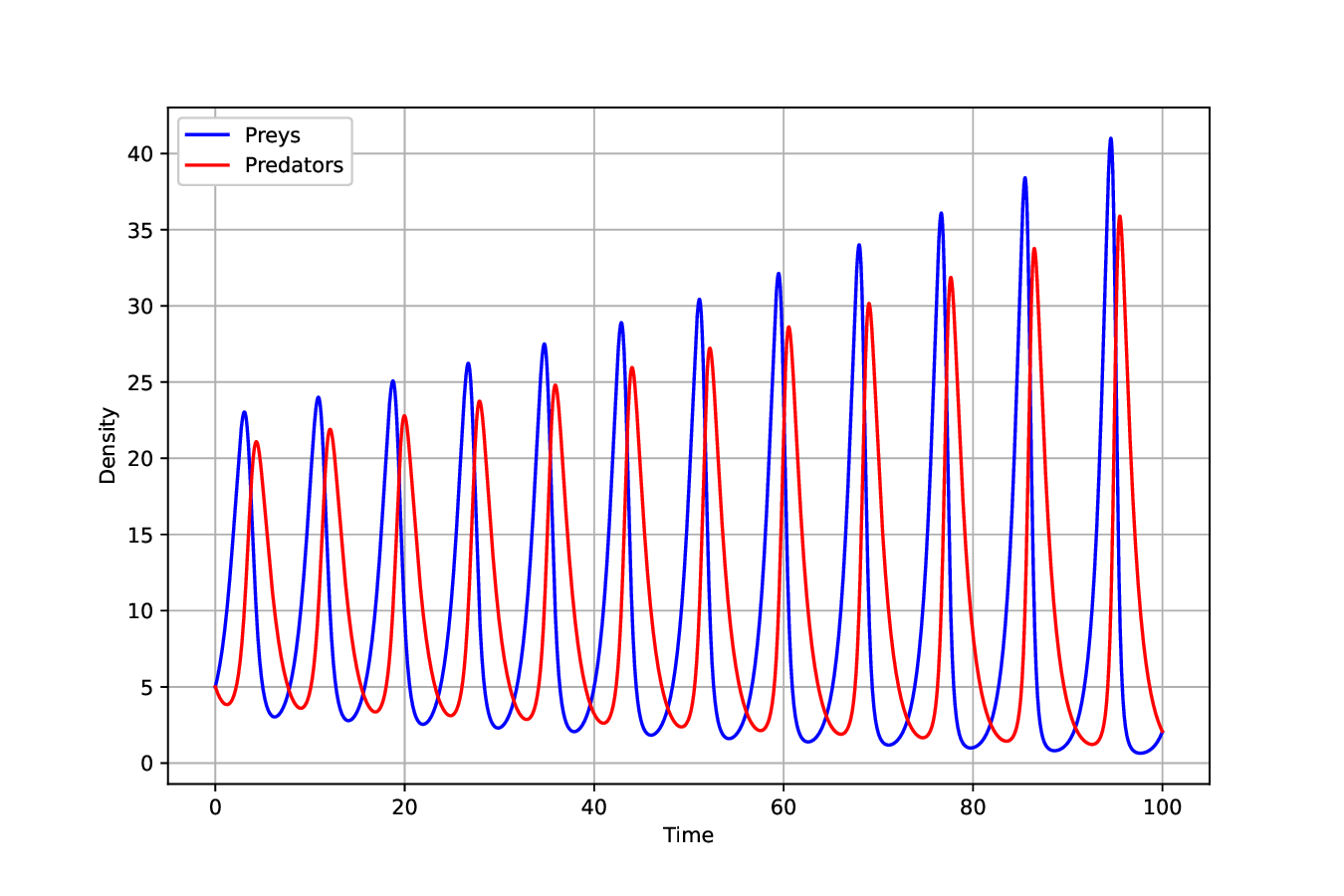} 
\caption{Oscillations of prey and predator densities for system (\ref{LST:euler}) 
with $\alpha = 1$, $\beta = 0.1$, $\gamma = 0.075$, $\delta = 0.75$, 
$h = 0.02$, and $x_0 = y_0 = 5$.}
\label{LST:oscilacoes_euler}
\end{figure}

In addition to the aforementioned dynamic inconsistency, caused by the progressive 
Euler method,\index{Euler's discretization} the considered discrete system 
also makes it possible to predict negative population densities, 
even when all parameters and initial conditions are taken positive. Furthermore, 
it is also possible to prove that, under some circumstances, 
negative solutions can return to positive values. 
Although mathematically possible, these two possibilities do not make any sense, 
neither in the context of the problem (the problem is defined only in $\mathbb{R}^2_+$) 
nor from an ecological point of view. However, as we shall prove analytically, 
and geometrically, under Euler's method both mentioned situations are indeed possible.
This shows the inconsistency of the discrete-time system (\ref{LST:euler}). Next
we study such situations in detail.

According to the orientation of the solution pointed out in Theorem~\ref{LST:trajectory_euler}, 
system (\ref{LST:euler}) can only predict negative solutions in two different cases and, 
in each of them, for only one of the variables.
\begin{enumerate}
\item Let $(x_i,y_i)$ be a point in region II of Figure~\ref{LST:regions}. 
Here we have $y_i > \frac{\alpha}{\beta}$, $h(\alpha x_i - \beta x_iy_i) < 0$, 
and $x_{i+1} < x_i$. From the first equation of system (\ref{LST:euler}), 
$x_{i+1}$ can assume a negative value if
\begin{equation*}
x_i < -h(\alpha x_i - \beta x_iy_i).
\end{equation*}	
Now, two situations can also occur. When the trajectory crosses the positive 
semi-axis $yy$, predicting a $x_{i+1} < 0$, this intersection can happen 
in such a way that $y_i$ remains greater than $\frac{\alpha}{\beta}$ 
or $y_i$ becomes less than that same value. We now note that, with $x_i < 0$, 
the first equation of the system can be rewritten as
\begin{equation*}
x_{i+1} = -x_i + h(-\alpha x_i + \beta x_iy_i) 
\quad \text{with} \quad x_i,y_i > 0.
\end{equation*}
\begin{itemize}
\item  If $y_i > \frac{\alpha}{\beta}$, then
we have $h(-\alpha x_i + \beta x_iy_i) > 0$.
Thus, $x_{i+1}$ can assume a positive value, as long as 
$h(-\alpha x_i + \beta x_iy_i) > x_i$. If this happens, 
the system enters region III of Figure~\ref{LST:regions}. 
Otherwise, the system goes outside the four mentioned regions, 
resulting in negative values for prey density.
\item If $y_i < \frac{\alpha}{\beta}$, 
then $h(\alpha x_i - \beta x_iy_i) < 0$, which leads to
\begin{equation*}
-x_i + h(-\alpha x_i + \beta x_i y_i) < 0 
\Rightarrow x_{i+1} < 0,
\end{equation*}
meaning that the system will go outside the 
four admissible regions.
\end{itemize}
\item Let $(x_i,y_i)$ be a point in region III of the 
Figure~\ref{LST:regions}. Here $x_i < \frac{\delta}{\gamma}$, 
$h(\gamma x_iy_i - \delta y_i) < 0$, and $y_{i+1} < y_i$. According 
to the second equation of system (\ref{LST:euler}), 
$y_{i+1}$ can assume a negative value if
\begin{equation*}
y_i < -h(\gamma x_iy_i - \delta y_i).
\end{equation*}
When the trajectory of the system crosses the positive semi-axis $xx$, 
obtaining $y_{i+1} < 0$, one can continue to have 
$x_i < \frac{\delta}{\gamma}$ or there can be a change 
in its value such that $x_i > \frac{\delta}{\gamma}$. Rewriting 
the second equation of system (\ref{LST:euler}), 
knowing that now $y_i < 0$, we obtain
\begin{equation*}
y_{i+1} = -y_i + h(-\gamma x_iy_i + \delta y_i) 
\quad \text{with} \quad x_i,y_i > 0.
\end{equation*}
\begin{itemize}
\item If $x_i < \frac{\delta}{\gamma}$, 
then $h(\gamma x_iy_i - \delta y_i) > 0$. Consequently, 
$y_{i+1}$ may be positive as long as $h(\gamma x_iy_i - \delta y_i) > y_i$. 
In this case, the trajectory of the system will enter in region IV. 
Otherwise, it will remain outside the four regions under study, 
with negative values for the density of predators.
\item If, on the other hand, $x_i > \frac{\delta}{\gamma}$, 
then $h(\gamma x_iy_i - \delta y_i) < 0$. Thus,
\begin{equation*}
-y_i + h(-\gamma x_iy_i + \delta y_i) < 0 
\Rightarrow y_{i+1} < 0,
\end{equation*}
which means that the system obtained by Euler's
method gives negative values for $y_{i+1}$, 
with values outside the four admissible regions.
\end{itemize} 
\end{enumerate}

By way of example, changing the value of $h$ from 0.02 to 0.03, it is possible 
to observe negative values for the variable $x_i$, as  
seen in Figure~\ref{LST:oscilacoes_euler_negativo}. In addition, 
it is verified that, after some time, the solutions that were previously 
negative return to positive values.

\begin{figure}[!ht]
\centering
\includegraphics[scale = 0.45]{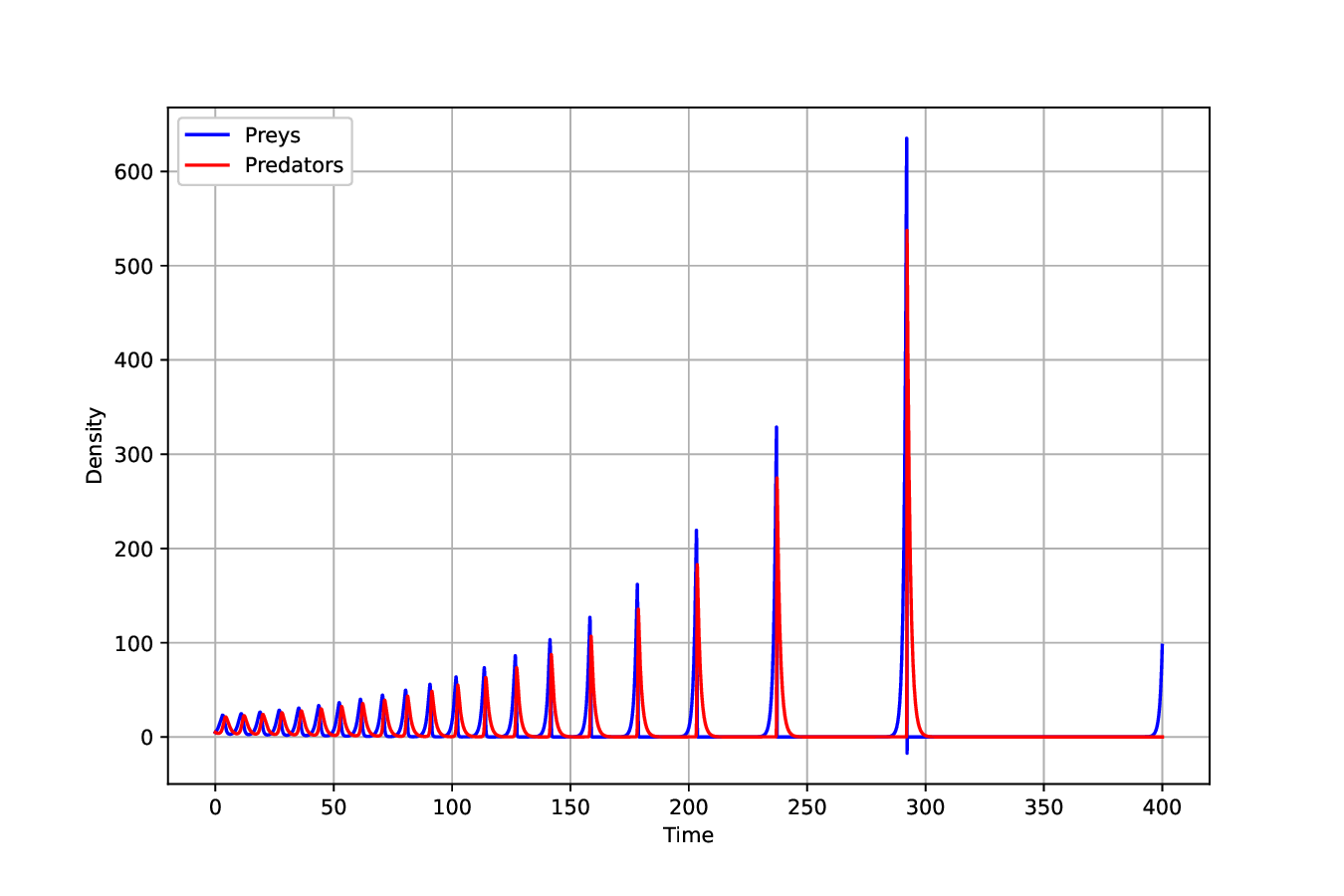} 
\caption{Oscillations of prey and predator densities for system (\ref{LST:euler}) 
with $\alpha = 1$, $\beta = 0.1$, $\gamma = 0.075$, $\delta = 0.75$, 
$h = 0.03$, and $x_0 = y_0 = 5$.}
\label{LST:oscilacoes_euler_negativo}
\end{figure}


\section{Discretization by Mickens' Method}
\label{LST:section03}

As seen in Section~\ref{LST:section02}, the progressive Euler method, 
when applied to the classical Lotka--Volterra system \eqref{LST:lotka_volterra}, 
has the particularity of losing the periodic\index{periodic} solutions, 
which correspond to closed curves in the phase space. In \cite{LST:mickens:lotka}, 
Mickens points out that the most likely reason for the loss of periodic\index{periodic} 
solutions is the fact that the Lotka--Volterra system is not structurally stable, 
i.e., a small perturbation in the equations of the system may change 
its topological properties. In particular, it can change the closed curves 
into ones that can spiral\index{spiral} into or out 
of the fixed point.\index{fixed point} It is known that 
the application of most classical numerical methods to a system 
with periodic solutions, transforms the original system into a very close one 
whose trajectories are not closed \cite{LST:sanz}. 

Here we intend to demonstrate that a nonstandard finite difference scheme, 
as generated according to the rules suggested by Mickens 
\cite{LST:mickens:paper},\index{Mickens' discretization} can be applied 
consistently to a structurally unstable dynamical system such 
as the one of Lotka--Volterra. In addition to proving that this scheme 
preserves the periodic solutions, it is also ensured that 
the positivity of the system is kept unchanged.

In \cite{LST:mickens:lotka}, Mickens suggests a 
discretization\index{Mickens' discretization} 
of the Lotka--Volterra model where, for simplicity, 
it is considered that all parameters -- 
$\alpha$, $\beta$, $\delta$, and $\gamma$ -- are equal to one. 
Here, the same strategy suggested by Mickens is followed, 
with the difference that the parameters are general,
assuming any value in $\mathbb{R}_+$. 

Following the rules stated by Mickens, 
the first-order derivatives are approximated by 
\begin{equation*}
\dot{x} \rightarrow \frac{x_{i+1} - x_i}{\phi}
\end{equation*}
and 
\begin{equation*}
\dot{y} \rightarrow \frac{y_{i+1} - y_i}{\phi},
\end{equation*}
where in both cases $\phi$ is such that 
$\phi(h) = h + \mathcal{O}(h^2)$.

Starting with the first equation of system (\ref{LST:lotka_volterra}), 
the linear and nonlinear terms are all substituted by nonlocal 
forms given by 
\begin{gather*}
\alpha x = 2\alpha x - x \rightarrow 2\alpha x_i - \alpha x_{i+1},\\ 
-\beta x y \rightarrow -\beta x_{i+1}y_i.
\end{gather*}
Thus, through the above substitutions, the first equation 
of system (\ref{LST:lotka_volterra}) can be rewritten as
\begin{equation*}
\frac{x_{i+1} - x_i}{\phi} 
= 2\alpha x_i - \alpha x_{i+1} - \beta x_{i+1}y_i,
\end{equation*}
which is equivalent to  
\begin{equation}
\label{LST:xmickens}
x_{i+1} = \frac{x_i(2\alpha\phi + 1)}{1 + \alpha\phi + \beta\phi y_i}.
\end{equation}

Regarding the second equation of system (\ref{LST:lotka_volterra}), 
the following substitutions are proposed:
\begin{gather*}
\gamma x y = 2\gamma x y - \gamma x y 
\rightarrow 2\gamma x_{i+1}y_i - \gamma x_{i+1}y_{i+1}, \\
- \delta y \rightarrow -\delta y_{i+1}.
\end{gather*}

Applying the two substitutions above, 
the second equation of the system is defined as
\begin{equation*}
\frac{y_{i+1} - y_i}{\phi} 
= 2\gamma x_{i+1}y_i - \gamma x_{i+1}y_{i+1} - \delta y_{i+1},
\end{equation*}
which is equivalent to
\begin{equation}
\label{LST:ymickens} 
y_{i+1} = \frac{y_i(2\gamma\phi x_{i+1} + 1)}{1 
+ \gamma \phi x_{i+1} + \delta\phi}. 
\end{equation}

Substituting (\ref{LST:xmickens}) into (\ref{LST:ymickens}), 
and joining both equations, we obtain the Lotka--Volterra model 
discretized by the Mickens method\index{Mickens' discretization} as
\begin{equation}
\label{LST:mickens}
\begin{cases}
x_{i+1} = \displaystyle \frac{x_i(2\alpha\phi 
+ 1)}{1 + \alpha\phi + \beta\phi y_i},\\[0.3cm]
y_{i+1} = \displaystyle \frac{2\gamma\phi x_i y_i(2\alpha\phi + 1) 
+ y_i(1 + \alpha\phi + \beta\phi y_i)}{(1 + \delta\phi)(1 + \alpha\phi 
+ \beta\phi y_i) + \gamma\phi x_i(2\alpha\phi + 1)},
\end{cases}
\end{equation}
which, as we shall show next, recovers the periodic\index{periodic} solutions 
and ensure that the positivity property of the Lotka--Volterra system 
is maintained. In concrete, through a simple analysis of the equations 
of system (\ref{LST:mickens}), it is clear that the Mickens method 
guarantees that the positivity property is maintained. Indeed, by choosing 
$(x_0,y_0) \in \mathbb{R}^2_+$, and as a consequence of all the parameters 
being positive, it is impossible to have negative values for any of the variables, 
since both equations will be quotients of strictly positive quantities. 

The fixed points\index{fixed point} of system (\ref{LST:mickens}) coincide 
with the ones of Sections~\ref{LST:section01} and \ref{LST:section02}: 
$p_1 = (0,0)$ and 
$p_2 = \left(\frac{\delta}{\gamma}, \frac{\alpha}{\beta}\right)$. 

Given the complexity of the system \eqref{LST:mickens} under study, we make use of 
the free open-source mathematics software system \textsf{SageMath}\index{\textsf{SageMath}} 
\cite{LST:sage} to analyze the nature of each one of the fixed points.\index{fixed point} 
For this purpose, we start by computing the Jacobian\index{Jacobian} matrix 
of the system (\ref{LST:mickens}) in an arbitrary point $(x,y)$. This matrix is given by 
\begin{equation}
\label{LST:jacobian_mickens}
Jf_{(x,y)} = 
\begin{pmatrix}
a & b \\
c & d
\end{pmatrix},
\end{equation} 
where
\begin{equation*}
\begin{aligned}
a =& \frac{2 \alpha \phi + 1}{\beta \phi y + \alpha \phi + 1},\\
b =& -\frac{{\left(2 \alpha \phi + 1\right)} 
\beta \phi x}{{\left(\beta \phi y + \alpha \phi + 1\right)}^{2}},\\
c =& \frac{2 {\left(2 \alpha \phi + 1\right)} \gamma 
\phi y}{{\left(2 \alpha \phi + 1\right)} \gamma \phi x 
+ {\left(\beta \phi y + \alpha \phi + 1\right)} {\left(\delta \phi + 1\right)}} \\
& -  \frac{{\left(2 {\left(2 \alpha \phi + 1\right)} \gamma \phi x y 
+ {\left(\beta \phi y + \alpha \phi + 1\right)} y\right)} {\left(2 \alpha \phi 
+ 1\right)} \gamma \phi}{{\left({\left(2 \alpha \phi + 1\right)} \gamma \phi x 
+ {\left(\beta \phi y + \alpha \phi + 1\right)} {\left(\delta \phi 
+ 1\right)}\right)}^{2}},\\
d =& -\frac{{\left(2 {\left(2 \alpha \phi + 1\right)} \gamma \phi x y 
+ {\left(\beta \phi y + \alpha \phi + 1\right)} y\right)} {\left(\delta \phi 
+ 1\right)} \beta \phi}{{\left({\left(2 \alpha \phi + 1\right)} \gamma \phi x 
+ {\left(\beta \phi y + \alpha \phi + 1\right)} {\left(\delta \phi 
+ 1\right)}\right)}^{2}} \\
&+ \frac{2 {\left(2 \alpha \phi + 1\right)} \gamma \phi x + 2 \beta \phi y 
+ \alpha \phi + 1}{{\left(2 \alpha \phi + 1\right)} \gamma \phi x 
+ {\left(\beta \phi y + \alpha \phi + 1\right)} {\left(\delta \phi + 1\right)}}.
\end{aligned}
\end{equation*}

\begin{theorem}
The fixed point\index{fixed point} $(0,0)$ 
of system \eqref{LST:mickens}
is a saddle point.\index{saddle point}
\end{theorem}

\begin{proof}
The Jacobian\index{Jacobian} matrix (\ref{LST:jacobian_mickens}) 
evaluated at the fixed point\index{fixed point} $(0,0)$ is
\begin{equation*}
Jf_{(0,0)} =
\begin{pmatrix}
\dfrac{2 \alpha \phi + 1}{\alpha \phi + 1} & 0 \\
0 & \dfrac{1}{\delta \phi + 1}
\end{pmatrix}, 
\end{equation*}
whose eigenvalues are $\lambda_1 = \dfrac{1}{\delta\phi + 1}$ 
and $\lambda_2 = \dfrac{2 \alpha \phi + 1}{\alpha \phi + 1}$. 
From these results, it is possible to draw the following conclusions:
\begin{itemize}
\item Since $\delta,\phi > 0$, it follows that $\delta\phi + 1 > 1$. 
Thus, $\lambda_1$ is always less than one, regardless of the values of 
$\delta$ and $\phi$. Moreover, by the positivity of 
the parameters, it is clear that $\lambda_1$ is always greater than zero. 
Thereby, $\lvert \lambda_1\rvert < 1$.
\item On the other hand, since $\alpha,\phi > 0 $, then $\alpha\phi + 1$ 
is always less than $2\alpha\phi + 1$. For this reason, $\lambda_2 > 1$, 
which leads to $\lvert \lambda_2 \rvert > 1$. 
\end{itemize}
Thus, $p_1$ is a saddle point\index{saddle point} and, therefore, unstable.
\end{proof}

In contrast, the Jacobian\index{Jacobian} matrix (\ref{LST:jacobian_cont}) 
evaluated at the coexistence equilibrium point $p_2$ is 
\begin{equation*}
Jf_{\left(\frac{\delta}{\gamma},\frac{\alpha}{\beta}\right)} 
=\begin{pmatrix}
1 & -\dfrac{\beta \delta \phi}{{\left(2 \alpha \phi + 1\right)} \gamma} \\
\dfrac{\alpha \gamma \phi}{2 \beta \delta \phi + \beta} 
& \dfrac{3 \alpha \delta \phi^{2} + 2 {\left(\alpha 
+ \delta\right)} \phi + 1}{4 \alpha \delta \phi^{2} 
+ 2 {\left(\alpha + \delta\right)} \phi + 1}
\end{pmatrix},
\end{equation*}
whose eigenvalues are complex conjugates
\begin{equation*}
\lambda = \frac{7 \alpha \delta \phi^{2} + 4 {\left(\alpha 
+ \delta\right)}\phi + 2 \pm i\phi \sqrt{\displaystyle 15 
\alpha^{2} \delta^{2} \phi^{2} + 4 \alpha \delta 
+ 8 {\left(\alpha^{2} \delta + \alpha \delta^{2}\right)} 
\phi}}{2 {\left(4 \alpha \delta \phi^{2} 
+ 2 {\left(\alpha + \delta\right)} \phi + 1\right)}}.
\end{equation*}
With the help of \textsf{SageMath}\index{\textsf{SageMath}}, 
it is easily verified that $|\lambda| = 1$, which means that 
the point $p_2$ is a center point in the linearized system, 
while nothing can be concluded regarding the stability\index{stability} 
for the nonlinear system at this equilibrium. However, it is possible 
to verify numerically that the orbits\index{orbit} are periodic,\index{periodic} 
corresponding to closed curves in the phase space, meaning that, 
at least for the indicated parameter values, $p_2$ is, in fact, a center. 
This effect can be seen in Figures~\ref{LST:vs_presas} and \ref{LST:vs_predadores}, 
which simultaneously show the results obtained here
and those obtained in the continuous case.

\begin{figure}[ht!]
\centering
\includegraphics[scale=0.45]{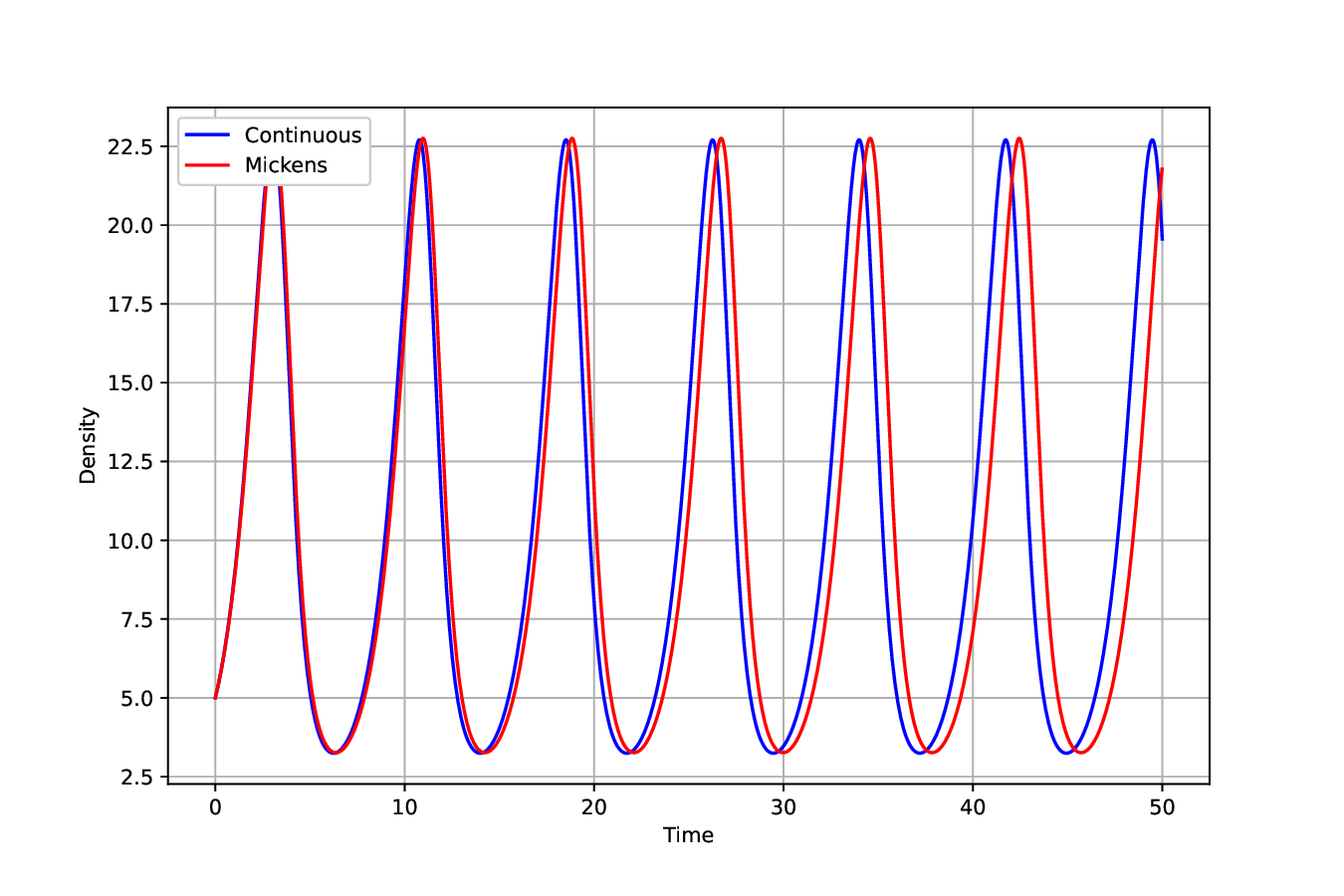}\\
\caption{Oscillations of preys for the system (\ref{LST:lotka_volterra}) 
versus system (\ref{LST:mickens}) with $\alpha = 1$, $\beta = 0.1$, 
$\gamma = 0.075$, and $\delta = 0.75$.}
\label{LST:vs_presas}
\end{figure}

\begin{figure}[ht!]
\centering
\includegraphics[scale=0.45]{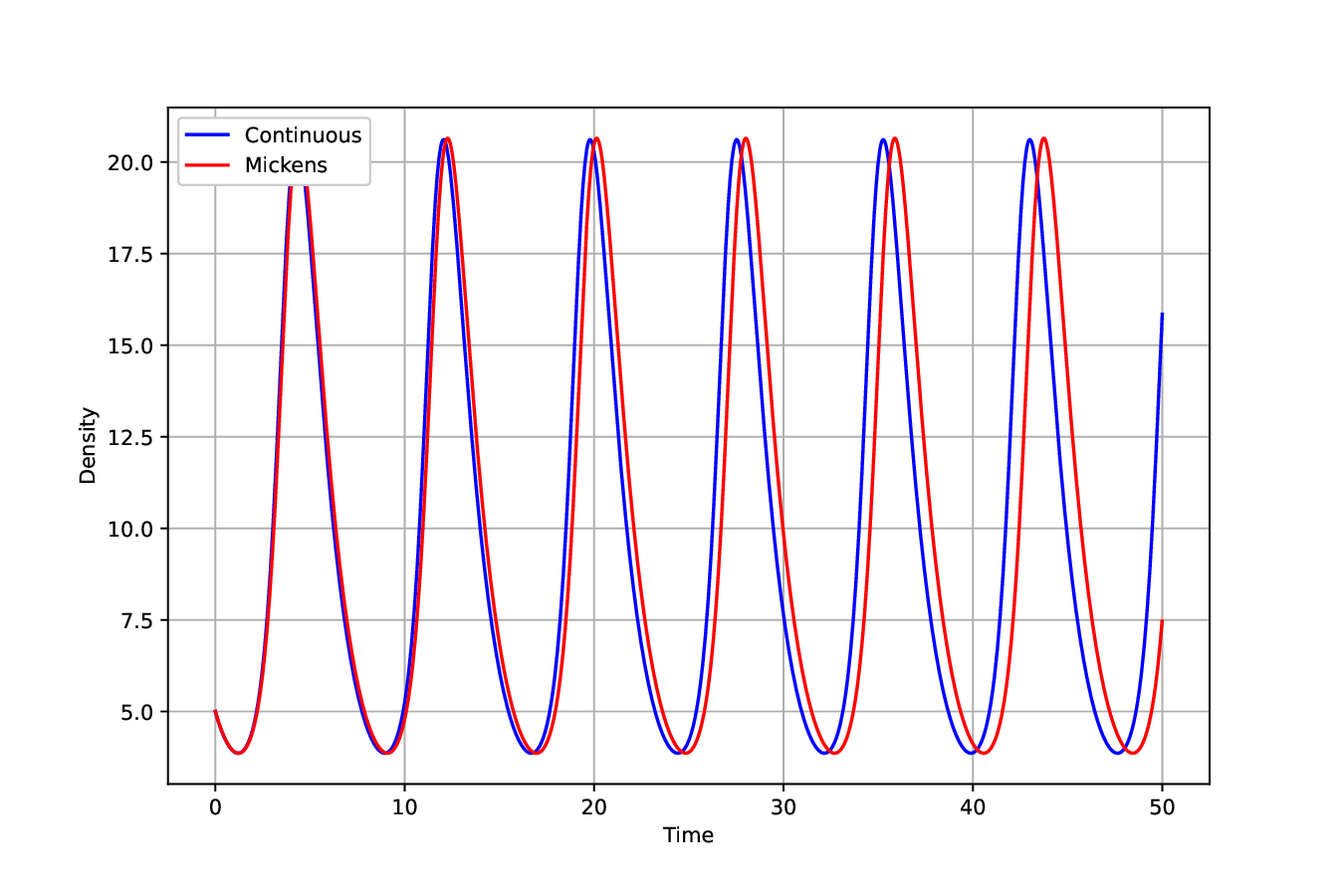}\\
\caption{Oscillations of predators for the system (\ref{LST:lotka_volterra}) 
versus system (\ref{LST:mickens}) with $\alpha = 1$, $\beta = 0.1$, 
$\gamma = 0.075$, and $\delta = 0.75$.}
\label{LST:vs_predadores}
\end{figure}

Despite the fact that one of Mickens' rules mention that a more complex expression 
should be used for the step function $\phi(h)$, it appears that all the results achieved 
are valid regardless of the expression used for $\phi(h)$. Accordingly, in our 
simulations we chose the simplest function given by $\phi(h) = h$. In particular, 
a step size given by $h = 0.01$ was considered. In both Figures~\ref{LST:vs_presas} and 
\ref{LST:vs_predadores}, it is observed that the periodic\index{periodic} 
oscillations\index{oscillations} of the discrete system practically overlap those 
of the original continuous system. Logically, the smaller the value of the chosen step $h$, 
the more superimposed the curves for each of the systems will be.

To complete the analysis of system (\ref{LST:mickens}), 
we end by proving that Theorem~\ref{LST:trajectory_euler} 
remains valid in this case, from which one can conclude that 
the direction of the trajectories of the Mickens' discrete system 
will continue to be counterclockwise.

Consider the first equation of system (\ref{LST:mickens}). 
Equivalently, one can write that
\begin{equation*}
\frac{x_{i+1}}{x_i} 
= \frac{2\alpha\phi + 1}{1 + \alpha\phi + \beta\phi y_i}.
\end{equation*}
\begin{itemize}
\item Let $y_i$ be a value that either belongs to regions I or II 
considered in Figure~\ref{LST:regions}. Then, we have 
$y_i > \frac{\alpha}{\beta}$. For this reason, 
\begin{equation*}
\alpha\phi < \beta\phi y_i \Rightarrow 2\alpha\phi 
< \alpha\phi + \beta\phi y_i.
\end{equation*}
Thus, 
\begin{equation*}
2\alpha\phi + 1 < 1 + \alpha\phi + \beta\phi y_i 
\Rightarrow \frac{x_{i+1}}{x_i} < 1 \Leftrightarrow x_{i+1} < x_i.
\end{equation*}
\item Now, let $y_i$ be a value that either belongs to regions III or IV. 
In both cases, we have $y_i < \frac{\alpha}{\beta}$, which leads to 
\begin{equation*}
\alpha\phi > \beta\phi y_i 
\Rightarrow 2\alpha\phi > \alpha\phi + \beta\phi y_i.
\end{equation*}  
Finally,
\begin{equation*}
2\alpha\phi + 1 > 1 + \alpha\phi + \beta\phi y_i 
\Rightarrow \frac{x_{i+1}}{x_i} > 1 \Leftrightarrow x_{i+1} > x_i.
\end{equation*}
\end{itemize} 

We now consider equation (\ref{LST:ymickens}) that is equivalent to
\begin{equation*}
\frac{y_{i+1}}{y_i} = \frac{2\gamma\phi x_{i+1} + 1}{1 
+ \gamma\phi x_{i+1} + \delta\phi}.
\end{equation*}
\begin{itemize}
\item Let $x_{i+1}$ be a value that either belongs to regions II or III 
of Figure~\ref{LST:regions}. In there we have $x_{i+1} < \frac{\delta}{\gamma}$, 
which is equivalent to $\delta > \gamma x_{i+1}$. In this way,
\begin{equation*}
\delta\phi > \gamma\phi x_{i+1} 
\Rightarrow \gamma\phi x_{i+1} + \delta\phi > 2\gamma\phi x_{i+1}.
\end{equation*}
Therefore,
\begin{equation*}
2\gamma\phi x_{i+1} + 1 < 1 + \gamma\phi x_{i+1} + \delta\phi 
\Rightarrow \frac{y_{i+1}}{y_i} < 1 \Leftrightarrow y_{i+1} < y_i.
\end{equation*}
\item On the other hand, let $x_{i+1}$ be a value that either belongs 
to regions I or IV. In this case, $x_{i+1} > \frac{\delta}{\gamma}$, 
which means that $\delta < \gamma x_{i+1}$, and we obtain
\begin{equation*}
\delta\phi < \gamma\phi x_{i+1} 
\Rightarrow \gamma\phi x_{i+1} + \delta\phi 
< 2\gamma\phi x_{i+1}.
\end{equation*}
Finally,
\begin{equation*}
2\gamma\phi x_{i+1} + 1 > 1 + \gamma\phi x_{i+1} + \delta\phi 
\Rightarrow \frac{y_{i+1}}{y_i} > 1 \Leftrightarrow y_{i+1} > y_i.
\end{equation*}
\end{itemize}

We conclude that Theorem~\ref{LST:trajectory_euler}	also holds
for system \eqref{LST:mickens}.


\section{Conclusion}

In this work our goal was to show that the choice of the numerical method 
for the discretization of a continuous dynamical system is crucial in order 
to obtain consistent results. It was proved that the progressive Euler 
method,\index{Euler's discretization} although appealing for its simplicity, 
is not able to deal with structurally unstable systems, making the solutions
of the classical Lotka--Volterra model, that should be closed curves 
in phase space, become spirals.\index{spiral} 
Furthermore, Euler's discretization does not take into account special 
fundamental properties of the systems, such as positivity. On the other hand, 
Mickens' method,\index{Mickens' discretization} despite generating an apparently 
more complex system, manages to guarantee that the qualitative behavior 
of the system, in a neighborhood of the fixed points,\index{fixed point} 
is identical to the one found in its continuous counterpart. Additionally, 
this method takes into account basic rules so that positivity is never compromised. 

Our conclusions open the possibility of applying 
Mickens' method\index{Mickens' discretization} to other structurally 
unstable dynamical systems\index{dynamical systems}  
of particular interest, recovering properties that may 
have been lost through different standard discretizations.
We also concluded that the Computer Algebra System 
\textsf{SageMath}\index{\textsf{SageMath}} 
is a strong tool that allows to do computations in a reliable way, 
serving as a good support when the systems under study are complex. 
In addition, it produces numerical simulations of good quality 
and in a very simple way. All the figures were generated
with \textsf{SageMath}\index{\textsf{SageMath}}.


\section*{Acknowledgments}

The authors were partially supported by 
the Portuguese Foundation for Science and Technology (FCT)
through the Center for Research and Development in Mathematics 
and Applications (CIDMA), projects UIDB/04106/2020 and UIDP/04106/2020.


\renewcommand{\bibname}{References}


\end{document}